\documentclass[11pt]{amsart}
\usepackage[margin=3.5cm]{geometry}

\usepackage{todonotes}
\usepackage{amssymb}
\usepackage{mathtools}  % \coloneqq :=
\usepackage{hyperref}

\newtheorem{theorem}{Theorem}[section]

\newtheorem{lemma}[theorem]{Lemma}

\theoremstyle{definition}

\newtheorem{problem}{Problem}

\numberwithin{equation}{section}

\newcommand{\dZ}{\mathbb Z}          
\DeclareMathOperator{\dens}{\,\mathrm dens}             % asymptotic density
\newcommand{\tL}{\mathtt 1}                             % typewriter 1
\newcommand{\tO}{\mathtt 0}                             % typewriter 0
\newcommand{\bt}{\mathbf t}                             % Thue--Morse sequence

\allowdisplaybreaks

\title{The sum-of-digits function on arithmetic progressions}
\author[Spiegelhofer]{Lukas Spiegelhofer}
\address{Institute of Discrete Mathematics and Geometry,
Vienna University of Technology,
Vienna, Austria}
\author[Stoll]{Thomas Stoll}
\address{
Institut \'Elie Cartan de Lorraine,
Universit\'e de Lorraine,
Vand\oe uvre-l\`es-Nancy, France}
\thanks{
The authors acknowledge support by the project MuDeRa,
which is a joint project between the FWF (Austrian Science Fund) and the ANR (Agence Nationale de la Recherche).
Moreover, the first author was supported by the FWF project F5502-N26, which is a part of the Special Research Program ``Quasi Monte Carlo methods: Theory and Applications'';
the second author was supported by the project ANR-18-CE40-0018.}

\subjclass[2010]{11A63, 11B25}
\begin{document}
\maketitle
\begin{abstract}
Let $s_2$ be the sum-of-digits function in base $2$, which returns the number of non-zero binary digits of a nonnegative integer $n$.
We study $s_2$ along arithmetic subsequences and show that --- up to a shift --- the set of $m$-tuples of integers that appear as an arithmetic subsequence of $s_2$ has full complexity.
\end{abstract}

\section{Results}
The binary sum-of-digits function $s_2$ is an elementary object studied in number theory.
It is defined by the equation
\[s_2(\varepsilon_\nu 2^\nu+\cdots+\varepsilon_0 2^0)=\varepsilon_\nu+\cdots+\varepsilon_0,\]
where $\varepsilon_i\in\{0,1\}$ for $0\leq i\leq\nu$.
Despite the simplicity of definition, the behaviour of $s_2$ on arithmetic progressions is not fully understood.
Cusick's conjecture on the sum-of-digits function
~\cite{DKS2016,S2019} concerns this area of research:
for an integer $t\geq 0$, we define the limit
\[c_t=\lim_{N\rightarrow \infty}\frac 1N\left\lvert\{n:0\leq n<N,s_2(n+t)\geq s_2(n)\}\right\rvert.
\]
(The limit exists, see for example B\'esineau~\cite{B1972}.
In fact, the set in this definition is periodic with period $2^k$ for some $k$.)
Cusick's conjecture states that
\begin{equation}\label{eqn_Cusick}
c_t>1/2
\end{equation} for all $t\geq 0$.
Drmota, Kauers, and the first author~\cite{DKS2016} proved that $c_t>1/2$ for \emph{almost all} $t$ in the sense of asymptotic density;
we also wish to note the works by Emme and Prikhod'ko~\cite{EP2017} and Emme and Hubert~\cite{EH2018,EH2018b}, and the recent partial result by the first author~\cite{S2019}.

In the current note, motivated by Cusick's conjecture, we are concerned with the $(m+1)$-tuple $\bigl(s_2(n),s_2(n+t),\ldots,s_2(n+mt)\bigr)$,
where $t\geq 0$ and $m\geq 1$ are integers.
We aim to understand the set of tuples that can occur, as $n$ and $t$ run.
In fact, our theorem states that, up to a shift, all tuples occur. 
\begin{theorem}\label{thm_main}  %\todo{Generalize this to base $q$.}
Assume that $k_1,\ldots,k_m\in\dZ$.
There exist $n$ and $t$ such that for $1\leq \ell\leq m$,
\[k_\ell=s_2(n+\ell t)-s_2(n).\]
\end{theorem}
This is a generalization of the statement that the Thue--Morse sequence $\mathbf t$ has full \emph{arithmetic complexity}, meaning that every finite word $\omega\in\{0,1\}^L$ occurs as an arithmetic subsequence of $\mathbf t$.
This was first proved in~\cite{AFF2003}
and also follows from M\"ullner and the first named author~\cite{MS2017}, and Konieczny~\cite{K2017}.

Theorem~\ref{thm_main} is not hard to prove for $m=1$.
We present three arguments leading to this fact.

\begin{enumerate}
\item Assume first that $k\geq 0$.
Set $n=2^{k+1}$ and $t=2^k-1$.
Then $s_2(n+t)=k+1$ and $s_2(n)=1$, yielding $k=s_2(n+t)-s_2(n)$.
If $k<0$, we set $n=2^{-k+1}-1$ and $t=1$.
Then $s_2(n)=-k+1$ and $s_2(n+t)=1$, which yields $s_2(n+t)-s_2(n)=k$.

\noindent
Alternatively, we may also write, as in the case $m=2$ presented below, 
$t=2^c-1$ and $n=2^{c-1}\left(2^a-1\right)$, for positive integers $a$ and $c$.
We obtain $s_2(n+t)=c$ and $s_2(n)=a$, and clearly the difference $c-a$ runs through all integers.

\item We have $s_2(n+1)-s_2(n)=1-\nu_2(n+1)\leq 1$, where $\nu_2(m)=\max\{k\geq 0:2^k\mid m\}$ for $m\geq 1$ is the $2$-adic valuation of $m$.
This formula follows by considering the number of $1$s with which the binary expansion of $n$ ends.
Since $s_2(2^\ell)=1$ and $s_2(2^{\ell+1}-1)=\ell+1$, we obtain the fact that $s_2(n)$ attains all values in $\{1,\ldots,\ell+1\}$ as $n$ varies in $\{2^\ell,\ldots,2^{\ell+1}-1\}$.
Let $k\in\mathbb Z$ be given and set $\ell=2\lvert k\rvert$.
Choose $n\in\{2^\ell,\ldots,2^{\ell+1}-1\}$ such that $s_2(n)=\lvert k\rvert +1$ and $n'\in\{2^{\ell+1},\ldots,2^{\ell+2}-1\}$ such that $s_2(n')=\lvert k\rvert+1+k$. Then $s_2(n')-s_2(n)=k$, which implies the statement.
\item Consider the densities 
\[\delta(k,t)=\lim_{N\rightarrow \infty}\frac 1N\left\lvert\{n:0\leq n<N,s_2(n+t)-s_2(n)=k\}\right\rvert
\]
(as it was the case for $c_t$, this asymptotic density exists~\cite{B1972}).
These quantities satisfy the following recurrence~\cite{DKS2016}:
%{{{eqn:simplified_recurrence
\begin{equation*}%\label{eqn:simplified_recurrence}
\begin{aligned}
\delta(k,1)&=\begin{cases}2^{k-2},&k\leq 1;\\0&\mbox{otherwise;}\end{cases}
\\
\delta(k,2t)&=\delta(k,t);\\
\delta(k,2t+1)&=\frac 12 \delta(k-1,t)+\frac 12\delta(k+1,t+1).
\end{aligned}
\end{equation*}
%}}}eqn:simplified_recurrence
From this, it is very easy to show that $\delta(k,t)>0$ for all $k\leq s_2(t)$.
For $k$ given, choose $t$ in such a way that $s_2(t)\geq k$; the positivity of the density $\delta(k,t)$ implies that there exists an $n$ such that $s_2(n+t)-s_2(n)=k$.
\end{enumerate}

For $m=2$, it is also possible to obtain the statement by elementary considerations:
consider integers $a,c\geq 1, b,d\geq 0$ and choose the integers $n$ and $t$ in such a way that the binary expansions look as follows:
\[
\begin{array}{rl@{\hspace{1pt}}l@{\hspace{1pt}}l@{\hspace{1pt}}l@{\hspace{1pt}}l}
  &\multicolumn{2}{l}{\overbrace{\hspace{1.1cm}}^a}&&\overbrace{\hspace{0.85cm}}^b\\[-0.4em]
n:&\tL\cdots\tL&\tL&\tO\cdots\tO  &\tL\cdots\tL&\tO\cdots\tO    \\
t:&                            &\tL&\tL\cdots\tL&\tO\cdots\tO &\tL\cdots\tL.\\[-1em]
  &&\multicolumn{2}{l}{\hspace{-0.3em}\underbrace{\hspace{1.1cm}}_c}&&\underbrace{\hspace{0.85cm}}_d\\
\end{array}
\]
The sums of digits of $n$, $n+t$ and $n+2t$ respectively are $a+b$, $b+c+d$ and $c+d$ respectively. By varying the variables, we can obtain the statement for all integers $k_1$ and $k_2$ such that $k_2\leq k_1$.
For the case $k_1<k_2$, we use the following configuration of the integers $n$ and $t$, where $a,d\geq 1$ and $c\geq 0$:
\[
\begin{array}{rl@{\hspace{1pt}}l@{\hspace{1pt}}l@{\hspace{1pt}}l@{\hspace{1pt}}l}
  &&\multicolumn{2}{l}{\hskip -0.5em\overbrace{\hspace{1.1cm}}^a}\\[-0.4em]
n:&&\tL\cdots\tL&\tL&\tO\cdots\tO\\
t:&\tL\cdots\tL&\tO\cdots\tO&\tL&\tL\cdots\tL.\\[-1em]
  &\underbrace{\hspace{0.85cm}}_c&&\multicolumn{2}{l}{\hskip -0.2em\underbrace{\hspace{1.1cm}}_d}\\
\end{array}
\]
The sums of digits of $n$, $n+t$ and $n+2t$ are $a$, $d$ and $c+d$ respectively, and we see that we obtain all pairs $(k_1,k_2)\in\mathbb Z^2$ such that $k_1\leq k_2$.

However, the method quickly experiences difficulties, as multiplication by $3$ is not a shift of the binary digits anymore.
While we believe that the case $m=3$ can be made work by some effort, a general principle is not apparent.
Therefore we choose a different approach.

We prove Theorem~\ref{thm_main} by induction on $m$, the cases $m=1,2$ having been discussed above.
Assume that $m\geq 3$ and let $k_1,\ldots,k_m\in\dZ$ be given.
By induction hypothesis, there exist $t_0$ and $n_0$ such that $k_\ell=s_2(n_0+\ell t_0)-s_2(n_0+(\ell-1)t_0)$ for $1\leq \ell<m$.
Set $k'_m=s_2(n_0+mt_0)-s_2(n_0+(m-1)t_0)$.
We are going to show that we may vary $k'_m$ by steps of $\pm 1$, thus yielding the full statement.

By concatenation of binary expansions, it is sufficient to show the following statement.
\begin{equation}\label{eqn_sufficient1}
\begin{aligned}&\mbox{There exist $t_1$, $n_1$ such that $s_2(n_1+\ell t_1)-s_2(n_1+(\ell-1)t_1)=0$ for $1\leq \ell<m$}\\
&\mbox{and $s_2(n_1+mt_1)-s_2(n_1+(m-1)t_1)=\pm 1$}.
\end{aligned}
\end{equation}
This concatenation is straightforward and summarized in the following lemma, which we will also use again in a moment.
\begin{lemma}\label{lem_block_decomposition}
Let $\ell\geq 1$, $m\geq 1$, $n_0,\ldots,n_{k-1}$ and $t_0,\ldots,t_{k-1}$ be nonnegative integers.
There exist nonnegative integers $n$ and $t$ such that
\[s_2(n+\ell t)-s_2(n+(\ell-1)t)=\sum_{0\leq j<k}\bigl(s_2(n_j+\ell t_j)-s_2(n_j+(\ell-1)t_j)\bigr)\]
for $1\leq \ell\leq m$.
\end{lemma}
\begin{proof}
The base case $k=1$ is trivial;
it is sufficient to prove the statement for $k=2$, the general case following easily from repeated application of this case.

Let $N$ be so large that $n_0+mt_0<2^N$, and set
$n=2^Nn_1+n_0$ and $t=2^Nt_1+t_0$. Since no carry propagation between the digits below and above $N$ occurs, we can add up the contribution of the two blocks in order to yield the statement.
\end{proof}
%In order to obtain~\eqref{eqn_sufficient1}, we consider $t_1$ whose binary expansion is of the form $(\tL\tO^{r-1})^{k-1}\tL$ for some $r,k\geq 1$, that is,
%$t_1=\sum_{0\leq j<k}2^{rj}$.
%The variable $r$ will be chosen large enough to avoid carry propagation between blocks.
We reduce the problem further, using this block representation again:
choose $t_j=1$ for all $0\leq j<k$;
it is sufficient to find  a $k\geq 1$ and nonnegative integers $n_j$ for $0\leq j<k$ such that
\begin{align}\label{eqn_sufficient}
\sum_{0\leq j<k}\left(s_2(n_j+\ell)-s_2(n_j+\ell-1)\right)
=\begin{cases}0,&\textrm{if }1\leq \ell< m;\\\pm 1,&\textrm{if }\ell=m.\end{cases}
\end{align}

In order to show~\eqref{eqn_sufficient}, we use the telescoping sum
\[
\sum_{a\leq j<a+2^L}g(j)=s_2(a+2^L)-s_2(a)=g\left(\left\lfloor a/2^L\right\rfloor\right),
\]
where
$g(j)=s_2(j+1)-s_2(j)$.
This representation yields for $1\leq \ell\leq m$, where $L$ is chosen such that $2^L\leq m<2^{L+1}$,
\begin{align*}
\sum_{2\cdot 2^L-m+\ell\leq j< 3\cdot 2^L-m+\ell}g(j)
&=g\left(2+\lfloor (-m+\ell)/2^L\rfloor\right)
=\begin{cases}
g(1)=0,&\textrm{if }1\leq \ell<m;\\
g(2)=1,&\textrm{if }\ell=m;
\end{cases}
\\
\sum_{2^L-m+\ell\leq j< 2\cdot 2^L-m+\ell}g(j)
&=g\left(1+\lfloor (-m+\ell)/2^L\rfloor\right)
=
\begin{cases}
g(0)=1,&\textrm{if }1\leq \ell<m;\\
g(1)=0,&\textrm{if }\ell=m;
\end{cases}
\\
\sum_{3\cdot 2^{L+1}+\ell\leq j< 4\cdot 2^{L+1}+\ell}g(j)
&=
g(3)=-1\mbox{ for }1\leq \ell\leq m.
\end{align*}
The first of these three identities yields the ``$+$''-part of~\eqref{eqn_sufficient} by choosing $k=2^L$ and $n_j=2\cdot 2^L-m+j$ for $0\leq j<k$.

The ``$-$''-part is obtained from the second and third identities:
by considering the disjoint union
$J=[2^L-m,2\cdot 2^L-m)\cup [3\cdot 2^{L+1},4\cdot 2^{L+1})$,
we have
\[\sum_{j\in J}\left(s_2(j+\ell)-s_2(j+\ell-1)\right)=
\begin{cases}
g(0)=0,&\textrm{if }1\leq \ell<m;\\
g(1)=-1,&\textrm{if }\ell=m.
\end{cases}
\]
The statement follows by merging the two intervals and choosing $n_j$ accordingly.
This finishes the proof of our theorem.

\section{Possible extensions}
From our proof, it is possible to actually construct integers $n$ and $t$ such that $s_2(n+\ell t)-s_2(n)=k_\ell$ for $1\leq \ell\leq m$.
In particular, this yields integers $n$ and $t$ such that $\bt_{n+\ell t}=\omega_\ell$ for $1\leq \ell\leq m$, where $(\omega_1,\ldots,\omega_m)\in\{0,1\}^m$
and $\bt$ is the Thue--Morse sequence on $\{0,1\}$.
(Note that we also used $\bt(2n+1)=1-\bt(n)$.)
This gives a constructive result concerning the problem of full arithmetic complexity of the Thue--Morse sequence considered in~\cite{AFF2003,K2017,MS2017}.

As an extension of the presented line of research, we are interested in the proportion of cases in which $s_2(n+\ell t)-s_2(n)=k_\ell$ occurs (for $1\leq \ell\leq m$).
For this, we define more generally
\[
\delta(\mathbf k,\boldsymbol\varepsilon,t)=\dens\left\{n:s_2(n+\ell t+\varepsilon_\ell)-s_2(n)=k_\ell\mbox{ for }1\leq \ell\leq m\right\},
\]
where $\mathbf k=(k_1,\ldots,k_m)\in\mathbb Z^m$ and $\boldsymbol\varepsilon=(\varepsilon_1,\ldots,\varepsilon_m)\in\mathbb N^m$.
This generalizes the array $\delta$ defined before.
As in the one-dimensional case, the densities in this definition actually exist,
and they satisfy the following recurrence relation:
\begin{align*}%\label{eqn_delta_recurrence_2t}
\delta(\mathbf k,\boldsymbol\varepsilon,2t)
&=\frac 12\dens\{n:s(2n+2\ell t+\varepsilon_\ell)-s(2n)=k_\ell\mbox{ for }1\leq \ell\leq m\}
\\&+\frac 12\dens\{n:s(2n+2\ell t+\varepsilon_\ell+1)-s(2n+1)=k_\ell\mbox{ for }1\leq \ell\leq m\}
\\&=\frac 12\delta(\mathbf k',\boldsymbol\varepsilon',t)
+\frac 12\delta(\mathbf k'',\boldsymbol\varepsilon'',t),
\end{align*}
where $k'_\ell=k_\ell-\varepsilon_\ell\bmod 2$, 
$k''_\ell=k_\ell+1-(\varepsilon_\ell+1) \bmod 2$,
$\varepsilon'_\ell=\lfloor \varepsilon_\ell/2\rfloor$ and
$\varepsilon''_\ell=\lfloor (\varepsilon_\ell+1)/2\rfloor$;
moreover,
\begin{align*}%\label{eqn_delta_recurrence_2t1}
\delta(\mathbf k,\boldsymbol\varepsilon,2t+1)
&=\frac 12\dens\{n:s(2n+2\ell t+\varepsilon_\ell+\ell)-s(2n)=k_\ell\mbox{ for }1\leq \ell\leq m\}
\\&+\frac 12\dens\{n:s(2n+2\ell t+\varepsilon_\ell+\ell+1)-s(2n+1)=k_\ell\mbox{ for }1\leq \ell\leq m\}
\\&= 
\frac 12\delta(\mathbf k',\boldsymbol\varepsilon',t)
+\frac 12\delta(\mathbf k'',\boldsymbol\varepsilon'',t),
\end{align*}
where $k'_\ell=k_\ell-(\varepsilon_\ell+\ell)\bmod 2$, 
$k''_\ell=k_\ell+1-(\varepsilon_\ell+\ell+1) \bmod 2$,
$\varepsilon'_\ell=\lfloor (\varepsilon_\ell+\ell)/2\rfloor$ and
$\varepsilon''_\ell=\lfloor (\varepsilon_\ell+\ell+1)/2\rfloor$.
This recurrence is the reason for the introduction of $\boldsymbol\varepsilon$.

This recurrence can be used to prove statements on the densities $\delta(\mathbf k,\boldsymbol\varepsilon,t)$.
On the one hand, we may ask for generalizations of Cusick's conjecture, relating the relative sizes of the values $s_2(n),s_2(n+t),\ldots,s_2(n+mt)$ to one another.
\begin{problem}\label{pb_cusick_multidim}
 Consider generalizations of Cusick's conjecture, proving for example, for many $t$, that
\[\dens\{n: s_2(n+t)\geq s_2(n)\mbox{ and }s_2(n+2t)\leq s_2(n+t)\}>1/4.\]
Moreover, show that the constant $1/4$ is optimal.
\end{problem}
On the other hand, we could ask for the overall shape of the $m$-dimensional
probability distribution defined by $\delta(\cdot,\boldsymbol\varepsilon,t)$.
\begin{problem}\label{pb_EH_multidim}
Prove a multidimensional generalization of the theorem by Emme and Hubert~\cite{EH2018}: for most $t$, the densities %$d(k_1,\ldots,k_m,t)=
$\dens\{n:s_2(n+\ell t)-s_2(n)=k_\ell\mbox{ for }1\leq \ell\leq m\}$ should define a probability distribution that is close to a multivariate Gaussian law. 
\end{problem}

We expect that nontrivial statements on both Problem~\ref{pb_cusick_multidim} and Problem~\ref{pb_EH_multidim}, at least for small $m$, can be obtained by extending the study of moments set forward by Emme and Hubert~\cite{EH2018}.
The transition to arbitrary dimensions $m$ however will necessitate new ideas.

\bibliographystyle{siam}
\bibliography{AP}

\end{document}